\documentclass[12pt, reqno]{amsart}
\usepackage{amsmath, amsthm, amscd, amsfonts, amssymb, graphicx, color}
\usepackage[bookmarksnumbered, colorlinks, plainpages]{hyperref}
\usepackage{color}
\input{mathrsfs.sty}

\usepackage[russian,english]{babel}
\Russian

\theoremstyle{plain}
\newtheorem{thm}{Theorem}[section]

\newtheorem{lemma}[thm]{Lemma}
\newtheorem{rem}[thm]{Remark}

\theoremstyle{definition}
\newtheorem{definition}[thm]{Definition}

\numberwithin{equation}{section}

\begin{document}
\title[Matrix KSGNS construction and a Radon--Nikodym type theorem]{Matrix KSGNS construction and a Radon--Nikodym type theorem }

\author{Mohammad Sal Moslehian}
\address{Department of Pure Mathematics, Center Of Excellence in Analysis on Algebraic Structures (CEAAS), Ferdowsi University of Mashhad, P. O. Box 1159, Mashhad 91775, Iran}
\email{moslehian@um.ac.ir and moslehian@member.ams.org}

\author{Anatoly Kusraev}
\address{Southern Mathematical Institute of the Russian Academy of Sciences\\
str. Markusa 22,
Vladikavkaz, 362027 Russia}
\email{kusraev@smath.ru}

\author{Marat Pliev}
\address{Southern Mathematical Institute of the Russian Academy of Sciences\\
str. Markusa 22,
Vladikavkaz, 362027 Russia;
RUDN University\\
6 Miklukho-Maklaya st, Moscow, 117198, Russia}
\email{maratpliev@gmail.com}



\begin{abstract}
In this paper, we introduce the concept of completely positive matrix of linear maps on Hilbert $A$-modules
over locally $C^{*}$-algebras and prove an analogue of Stinespring theorem for it. We show that any two
minimal Stinespring representations for such matrices are unitarily equivalent. Finally, we prove an analogue of the Radon--Nikodym theorem for this type of completely positive $n\times n$ matrices.
\end{abstract}

\maketitle

\section{Introduction and preliminaries}

The study of completely positive linear maps is motivated by applications of the theory of completely positive linear maps to quantum
information theory, where operator valued completely positive linear maps on topological algebras with an involution are used as a mathematical model for quantum operations and quantum probability. In the article we deal with rather a wide class of algebras, which is called the class of locally $C^{*}$-algebras.
This class was first introduced by Inoue \cite{In}. Later Phillips \cite{Ph} showed that locally $C^{*}$-algebras can be considered as inverse limits of $C^{*}$-algebras.
In recent years, Joi\c{t}a has provided a deep theory of Hilbert $\mathscr{A}$-modules over locally $C^{*}$-algebras; see \cite{J-00,J-0}.
For basic information about locally $C^{*}$-algebras and Hilbert modules over them the reader is referred to \cite{F,J-00}.

A linear map $\varphi:\mathscr{A}\to \mathscr{B}$ between locally $C^*$-algebras is said to be {\it positive}, if $\varphi(a^*a)\geq 0$ for all $a\in \mathscr{A}$. A {\it completely positive map} $\varphi:\mathscr{A}\to \mathscr{B}$ of locally $C^*$-algebras is a linear map such that $\varphi_n: \mathbb{M}_{n}(\mathscr{A})\to \mathbb{M}_{n}(\mathscr{B})$ defined by $\varphi_n\left((a_{ij}))_{i,j=1}^{n}\right)=\left(\varphi(a_{ij})\right)_{i,j=1}^{n}$ is positive. Stinespring \cite{St} showed that a completely positive linear map $\varphi$ from $\mathscr{A}$ to the $C^*$-algebra $\mathscr{L}(\mathcal{H})$ of all bounded linear operators acting on a Hilbert space $\mathcal{H}$ is of the form $\varphi(\cdot)=S^*\pi(\cdot)S$, where $\pi$ is a $*$-representation of $\mathscr{A}$ on a Hilbert space $\mathcal{K}$ and $S$ is a bounded linear operator from $\mathcal{H}$ to $\mathcal{K}$. Nowadays, the theory of completely positive linear maps on Hilbert and Krein $A$-modules is a vast area of the modern analysis (see \cite{B, HH, JJM, J-1, J-2, MP, PT, S}).

Throughout the paper, we identify any $n\times n$ matrix $(\varphi_{ij})_{i,j=1}^{n}$, whose entries are linear maps from $\mathscr{A}$ to $\mathscr{B}$, with the  linear map $[\varphi]:\mathbb{M}_{n}(\mathscr{A})\to \mathbb{M}_{n}(\mathscr{B})$
defined by
$$
[\varphi]((a_{ij})_{i,j=1}^{n})=(\varphi_{ij}(a_{ij}))_{i,j=1}^{n}\,.
$$
If $[\varphi]$ is a completely positive linear map from $\mathbb{M}_{n}(\mathscr{A})$ to $\mathbb{M}_{n}(\mathscr{B})$, then we say that $[\varphi]$ is a {\it completely $n$-positive linear map} from $\mathscr{A}$ to $\mathscr{B}$. In this case, the map $\varphi_{ii}$ is clearly a completely positive linear map from $\mathscr{A}$ to $\mathscr{B}$ for each $1\leq i\leq n$; see \cite{SUE}.
In this note, we consider a matrix version of KSGNS construction and prove a Radon--Nikodym type theorem for completely positive $n\times n$ matrices of linear maps on Hilbert $\mathscr{A}$-modules.

A Hilbert $\mathscr{A}$-module $\mathcal{M}$ over a locally $C^*$-algebra $\mathscr{A}$ is a right $\mathscr{A}$-module, equipped with an $\mathscr{A}$-valued inner product $\langle\cdot,\cdot\rangle$ that is
$\Bbb{C}$-linear and $\mathscr{A}$-linear in the second variable and $\langle y,x\rangle=\langle x,y\rangle^*\,\,(x,y \in \mathcal{M})$ such that $\mathscr{A}$ is complete under the topology generated by the family of seminorms $\|x\|_{\alpha}=\|\langle x,x\rangle\|_{\alpha}^{\frac{1}{2}}$, $\alpha\in\Delta$. If the closed two-sided ideal
$\langle \mathcal{M},\mathcal{M}\rangle$ of $\mathscr{A}$ generated by $\{\langle x,y\rangle:\,x,y\in \mathcal{M}\}$ coincides with $\mathscr{A}$, we say that
$\mathcal{M}$ is full.

We need some auxiliary information about locally Hilbert spaces. The reader can find some detailed description in \cite[Chap.~2]{F}.

\begin{definition}
Let $\Delta$ be a upward directed set, $\mathcal{H}_{\alpha}$ be a Hilbert space for every $\alpha\in\Delta$, $\mathcal{H}_{\alpha}$ be a
closed subspace of $\mathcal{H}_{\beta}$, $\alpha\leq\beta$ and $\mathcal{H}=\lim\limits_{\longrightarrow}\mathcal{H}_{\alpha}=\bigcup_{\alpha\in\Delta}\mathcal{H}_{\alpha}$. Endow the vector space $\mathcal{H}$ with an inductive limit topology, that is the finest locally convex topology making the maps $i_{\alpha}:\mathcal{H}_{\alpha}\to \mathcal{H}$ continuous. Then the topological vector space $\mathcal{H}$ is called a {\it locally Hilbert space}.
\end{definition}
By $i_{\alpha}$ and $i_{\alpha\beta}$ we denote the continuous embeddings $i_{\alpha}:\mathcal{H}_{\alpha}\to \mathcal{H}$
and $i_{\alpha,\beta}:\mathcal{H}_{\alpha}\to \mathcal{H}_{\beta}$, respectively. Elements of $\mathcal{H}$ can be considered as $(\xi_{\alpha})_{\alpha\in\Delta}$, where $\xi_{\alpha}\in \mathcal{H}_{\alpha}$ and  $\xi_{\beta}=i_{\alpha\beta}\xi_{\alpha}$ for all $\alpha,\beta\in\Delta$, $\alpha\leq\beta$.
Let $\mathcal{H}$ and $\mathcal{K}$ be locally Hilbert spaces with the same index set $\Delta$ and let $T$ be a linear operator
from $\mathcal{H}$ to $\mathcal{K}$. Since $\mathcal{H}_{\alpha}$ ($\mathcal{K}_{\alpha}$) is a closed subspace of $\mathcal{H}_{\beta}$ ($\mathcal{K}_{\beta}$) we can define the projection $P_{\alpha\beta}:\mathcal{H}_{\beta}\to\mathcal{H}_{\alpha}$
($R_{\alpha\beta}:\mathcal{K}_{\beta}\to\mathcal{K}_{\alpha}$). Consider linear operator $T:\mathcal{H}\to\mathcal{K}$ such that
$R_{\alpha\beta}T=P_{\alpha\beta}T$ for every $\alpha\leq\beta$.
Put $T_{\alpha}:=T|_{\mathcal{H}_{\alpha}}$.
Recall that for two Hilbert spaces $\mathcal{H}_{\alpha}$ and $\mathcal{K}_{\alpha}$, the Banach space of all bounded linear
operators from $\mathcal{H}_{\alpha}$ to $\mathcal{K}_{\alpha}$ is denoted by $\mathscr{L}(\mathcal{H}_{\alpha},\mathcal{K}_{\alpha})$.
Put
$$
\mathscr{L}(\mathcal{H},\mathcal{K}):=\{T:\mathcal{H}\to\mathcal{K}:\,T=\lim\limits_{\longrightarrow}T_{\alpha}:T_{\alpha}\in \mathscr{L}(\mathcal{H}_{\alpha},\mathcal{K}_{\alpha})\}
$$
Let $T=\lim\limits_{\longrightarrow}T_{\alpha}\in\mathscr{L}(\mathcal{H},\mathcal{K})$. Consider the adjoint $T^{*}_{\alpha}$ of
$T_{\alpha}\in\mathscr{L}(\mathcal{H}_{\alpha},\mathcal{K}_{\alpha})$, $\alpha\in\Delta$ and let $\alpha\leq\beta$ in $\Delta$,
$x\in \mathcal{H}_{\alpha}$, $y\in \mathcal{H}_{\alpha}$. Then
\begin{gather*}
\langle T^{*}_{\beta}x,y\rangle_{\beta}=\langle x,T_{\beta}y\rangle_{\beta}=\langle x,T_{\alpha}y\rangle_{\alpha}=
\langle T^{*}_{\alpha}x,y\rangle_{\alpha}.
\end{gather*}
Thus we have $T^{*}_{\beta}|_{\mathcal{K}_{\alpha}}=T^{*}_{\alpha}$ for every $\alpha,\beta\in\Delta$, $\alpha\leq\beta$.
Then there exists a unique element $T^{*}\in\mathscr{L}(\mathcal{K},\mathcal{H})$ with
$$
T^{*}=\lim\limits_{\longrightarrow}T^{*}_{\alpha}\,\,\,\text{such that}\,\,\,
T^{*}|_{\mathcal{K}_{\alpha}}=T^{*}_{\alpha},\,\alpha\in\Delta.
$$
The operator $T^{*}$ is called the {\it adjoint} of $T$. It is worth to note that $\mathscr{L}(\mathcal{H},\mathcal{K})$
has a structure of a Hilbert $\mathscr{L}(\mathcal{H})$-module over $\mathscr{L}(\mathcal{H})$ with the usual composition as the right module action and the inner product given by $\langle T_{1},T_{2}\rangle=T_{1}^*T_{2}$ for $T_{1},T_{2}\in \mathscr{L}(\mathcal{H},\mathcal{K})$. The topology on $\mathscr{L}(\mathcal{H},\mathcal{K})$ is generated by the family of seminorms $\|\cdot\|_{\alpha}$, $\alpha\in\Delta$, where
$$
\|T\|_{\alpha}:=\|T^{*}_{\alpha}T_{\alpha}\|_{\alpha},\,\,\alpha\in\Delta.
$$
\begin{definition}
A representation of Hilbert $\mathscr{A}$-module $\mathcal{M}$ on locally Hilbert spaces $\mathcal{H}$ and $\mathcal{K}$ is a continuous linear map $\Pi:\mathcal{M}\to \mathscr{L}(\mathcal{H},\mathcal{K})$
with the property that there is a continuous $*$-representation $\pi$ of $\mathscr{A}$ on the locally Hilbert space $\mathcal{H}$ such that
$$
\langle \Pi(x),\Pi(y)\rangle=\pi(\langle x,y\rangle)
$$
for all $x,y\in\mathcal{M}$. A representation $\Pi:\mathcal{M}\to \mathscr{L}(\mathcal{H},\mathcal{K})$ of $\mathcal{M}$ is nondegenerate if $[\Pi(\mathcal{M})(\mathcal{H})]=\mathcal{K}$
and $[\Pi(\mathcal{M})^*(\mathcal{K})]=\mathcal{H}$ (throughout the paper, $[Y]$ denotes the closed subspace of a topological vector space $Z$ generated by a subset $Y$ of $Z$).
\end{definition}

If $\mathcal{M}$ is full, then the $*$-representation $\pi$ associated to $\Pi$ is evidently unique.

A linear map $\Phi:\mathcal{M}\to \mathscr{L}(\mathcal{H},\mathcal{K})$ is called \textit{completely positive} on $\mathcal{M}$ if there is a completely positive linear map
$\varphi:\mathscr{A}\to \mathscr{L}(\mathcal{H})$ such that
$$
\langle \Phi(x),\Phi(y)\rangle=\varphi(\langle x,y\rangle)
$$
for all $x,y\in\mathcal{M}$.

\section{The matrix KSGNS construction}

Let $\mathcal{M}$ be a Hilbert module over $\mathscr{A}$ and let $\mathcal{H},\mathcal{K}$ be locally Hilbert spaces.
In all considerations below, we assume that for a locally $C^{*}$-algebra $\mathscr{A}=\lim\limits_{\longleftarrow}\mathscr{A}_{\alpha}$ the set $S(\mathscr{A})$ of all continuous $C^{*}$ seminorms on $\mathscr{A}$ coincides
with the indexing set $\Delta$, where
$\mathcal{H}=\bigcup_{\alpha\in\Delta}\mathcal{H}_{\alpha}$.

Let $\Phi_{ij}:\mathcal{M}\to \mathscr{L}(\mathcal{H},\mathcal{K})\,\,i,j\in\{1,\dots,n\}$ be given continuous linear maps.

\begin{definition}\label{CP}
An $n\times n$ matrix $[\Phi]:=(\Phi_{ij})_{i,j=1}^{n}$ of continuous linear maps $\Phi_{ij}:\mathcal{M}\to \mathscr{L}(\mathcal{H},\mathcal{K})$ is called \textit{completely positive}, if there exists a continuous completely $n$-positive map $[\varphi]$ from $\mathscr{A}$ to $\mathscr{L}(\mathcal{H})$ such that
\begin{gather}
\langle[\Phi](x),[\Phi](y)\rangle:=(\Phi_{ij}(x)_{i,j=1}^{n})^{*}(\Phi_{ij}(y))_{i,j=1}^{n}=(\varphi_{ij}\langle x,y\rangle)_{i,j=1}^{n}
\end{gather}
for every $x,y\in \mathcal{M}$. Then we say that $[\Phi]$ is a $[\varphi]$-completely positive $n\times n$ matrix.\\
The equality above needs clarification.
By $(\Phi_{ij}(x)_{i,j=1}^{n})^{*}$ we mean $(\Phi_{ji}(x)^{*})_{i,j=1}^{n}$, where $\Phi_{ji}(x)^{*}$
is the adjoint of $\Phi_{ji}(x)$. Thus, in equality~\ref{CP}, we have
$$
\varphi_{ij}\langle x,y\rangle=\sum_{r=1}^{n}\Phi_{ri}(x)^{*}\Phi_{rj}(y).
$$
By $\mathcal{CP}_{n}(\mathcal{M},\mathscr{L}(\mathcal{H},\mathcal{K}))$ we denote the set of
all completely positive $n\times n$ matrices from $\mathcal{M}$ to $\mathscr{L}(\mathcal{H},\mathcal{K})$.
\end{definition}
\begin{thm}\label{KGNS}
Let $\mathscr{A}=\lim\limits_{\longleftarrow}\mathscr{A}_{\alpha}$ be a unital locally $C^*$-algebra and $\mathcal{M}$ be a Hilbert $\mathscr{A}$-module over $\mathscr{A}$. Let $[\varphi]:\mathscr{A}\to \mathscr{L}(\mathcal{H})$ be a continuous completely $n$-positive map corresponded to a matrix $(\varphi_{ij})_{i,j=1}^{n}$ of linear maps from $\mathscr{A}$ to $\mathscr{L}(\mathcal{H})$, and $[\Phi]=(\Phi_{ij})_{i,j=1}^{n}$, $\Phi_{ij}:\mathcal{M}\to \mathscr{L}(\mathcal{H},\mathcal{K}),\,i,j\in\{1,\dots,n\}$ be a $[\varphi]$-completely positive $n\times n$ matrix. Then there exists a data $(\pi^{\Phi},\mathcal{H}^{\Phi},\mathcal{K}^{\Phi},S_{1}^{\Phi},\dots, S_{n}^{\Phi},W_{1}^{\Phi},\dots, W_{n}^{\Phi})$, where
$\mathcal{H}^{\Phi},\mathcal{K}^{\Phi}$ are locally Hilbert spaces; $\pi^{\Phi}:\mathcal{M}\to \mathscr{L}(\mathcal{H}^{\Phi},\mathcal{K}^{\Phi})$ is a representation of the module $\mathcal{M}$ on the locally Hilbert spaces $\mathcal{H}^{\Phi}$ and $\mathcal{K}^{\Phi}$, which associated with the continuous $*$-homomorphism $\pi^{\varphi}:\mathscr{A}\to \mathscr{L}(\mathcal{H}^{\Phi})$, the maps $S_{i}^{\Phi}:\mathcal{H}\to \mathcal{H}^{\Phi}$ and $W_{i}^{\Phi}:\mathcal{K}\to \mathcal{K}^{\Phi}$ are continuous linear operators such that:
\begin{enumerate}
\item~$\varphi_{ij}(a)=(S_{i}^{\Phi})^*\pi^{\varphi}(a)S_{j}^{\Phi}$ for every $a\in \mathscr{A}$, $i,j\in\{1,\dots,n\}$;
\item~$\Phi_{ij}(x)=(W_{i}^{\Phi})^*\pi^{\Phi}(x)S_{j}^{\Phi}$ for every $x\in \mathcal{M}$, $i\in\{1,\dots,n\}$;
\item~$\mathcal{H}^{\Phi}=[\{\pi^{\varphi}(\mathscr{A})S_{i}^{\Phi}(\mathcal{H}):\,i,j=1,\dots,n\}]$;
\item~$\mathcal{K}^{\Phi}=[\{\pi^{\Phi}(\mathcal{M})S_{i}^{\Phi}(\mathcal{H}):\,i=1,\dots,n\}]$.
\end{enumerate}
\end{thm}
\begin{proof}
At first, we prove the existence of $\pi^{\varphi}$, $\mathcal{H}^{\Phi}$ and $S_{1}^{\Phi},\dots,S_{n}^{\Phi}$. We denote by $(\mathscr{A}_{\alpha}\otimes_{\text{alg}}\mathcal{H}_{\alpha})^{n}$ the direct sum of $n$ copies of the algebraic tensor
product $\mathscr{A}_{\alpha}\otimes_{\text{alg}} \mathcal{H}_{\alpha}$, $\alpha\in \Delta$. It is
not difficult to verify that $(\mathscr{A}_{\alpha}\otimes_{\text{alg}}\mathcal{H}_{\alpha})^{n}$ is a semi-linear product space under
$$
\big\langle\sum_{s=1}^{m}(a_{is}\otimes\xi_{is})_{i=1}^{n},\sum_{t=1}^{l}(b_{jt}\otimes\eta_{jt})_{j=1}^{n}\big\rangle_{\alpha0}=
\sum_{s,t=1}^{m,l}\sum_{i,j=1}^{n}\langle\xi_{is},\varphi_{ij}(a_{is}^{*}b_{jt})\eta_{jt}\rangle_{\alpha}\,.
$$
Put $M_{\alpha}:=\{\zeta\in(\mathscr{A}_{\alpha}\otimes_{\text{alg}}\mathcal{H}_{\alpha})^{n}):\langle\zeta,\zeta\rangle_{\alpha0}=0\}$. Employing the Cauchy--Schwarz inequality, we observe that $M_{\alpha}$
is a subspace of $(\mathscr{A}_{\alpha}\otimes_{\text{alg}}\mathcal{H}_{\alpha})^{n}$. Then
$(\mathscr{A}_{\alpha}\otimes_{\text{alg}}\mathcal{H}_{\alpha})^{n}/M_{\alpha}$ becomes a pre-Hilbert space with the inner product defined by
$$
\langle\zeta_{1}+M_{\alpha},\zeta_{2}+M_{\alpha}\rangle:=\langle\zeta_{1},\zeta_{2}\rangle_{\alpha0}
$$
for every $\alpha\in\Delta$.
The completion of $(\mathscr{A}_{\alpha}\otimes_{\text{alg}}\mathcal{H}_{\alpha})^{n}/M_{\alpha}$ with respect to the topology induced by this inner product is denoted by $\mathcal{H}_{\alpha}^{\Phi}$. Observe that if $M_{\alpha}\supset M_{\beta}$ and $\alpha,\beta\in\Delta$ and therefore
there is continuous embeddings  $j_{\alpha\beta}:\mathcal{H}_{\alpha}^{\Phi}\to\mathcal{H}_{\beta}^{\Phi}$. Put $\mathcal{H}^{\Phi}=\lim\limits_{\longrightarrow}\mathcal{H}_{\alpha}^{\Phi}$.
We denote by $\xi_{i}$ the element in $\lim\limits_{\longrightarrow}\big((\mathscr{A}_{\alpha}\otimes \mathcal{H}_{\alpha})^{n}/M_{\alpha}\big)$, whose $i^{\text{th}}$ component is $(1_{\alpha})_{\alpha\in\Delta}\otimes(\xi_{\alpha})_{\alpha\in\Delta}$ and all other components are $0$. Now we can define a linear map $S_{i}^{\Phi}:\mathcal{H}\rightarrow \mathcal{H}^{\Phi}$
$$
S_{i}(\xi)=\xi_{i}.
$$
Let us denote by $\xi_{a,i}$ the element in $\lim\limits_{\longrightarrow}\big((\mathscr{A}_{\alpha}\otimes_{\text{alg}}\mathcal{H}_{\alpha})^{n}/M_{\alpha}\big)$, whose $i^{\text{th}}$ component is $a\otimes\xi$ and all other components are $0$. Let $a\in \mathscr{A}$. Consider the linear map $\pi^{\varphi}(a):\lim\limits_{\longrightarrow}(\mathscr{A}_{\alpha}\otimes_{\text{alg}}\mathcal{H}_{\alpha})^{n}\rightarrow \lim\limits_{\longrightarrow}(\mathscr{A}_{\alpha}\otimes_{\text{alg}}\mathcal{H}_{\alpha})^{n}$ defined by
$$
\pi^{\varphi}(a)(a_{i}\otimes\xi_{i})_{i=1}^{n}=(aa_{i}\otimes\xi_{i})_{i=1}^{n}.
$$
The linear map $\pi^{\varphi}(a)$ can be extended by
linearity and continuity to a linear map, denoted also by $\pi^{\varphi}(a)$, from $\mathcal{H}^{\Phi}$ to $\mathcal{H}^{\Phi}$.
The fact that $\pi^{\varphi}(a)$ is a representation of $\mathscr{A}$ on $\mathscr{L}(\mathcal{H}^{\Phi})$ is showed in the same manner as in the proof of Theorem~3.3.2 of \cite{J-0}.
It is easy to check
that $\pi^{\varphi}(a_{i})S_{i}^{\Phi}\xi_{i}=\xi_{i,a}+\lim\limits_{\longrightarrow}M_{\alpha}$. Therefore the subspace of $\mathcal{H}^{\Phi}$ generated by $\pi^{\varphi}(a_{i})S_{i}^{\Phi}\xi_{i},\,i\in\{1,\dots,n\},\,\xi_{i}\in \mathcal{H}$, $a_{i}\in \mathscr{A}$ is exactly $\lim\limits_{\longrightarrow}\big((\mathscr{A}\otimes_{\text{alg}}\mathcal{H})^{n}/M_{\alpha}\big)$.

Let $\mathcal{K}^{\Phi}:=(\mathcal{K}_{1}^{\Phi},\dots,\mathcal{K}_{n}^{\Phi})$, where

$\mathcal{K}_{i}^{\Phi}:=\left[\bigcup_{j=1}^{n}\Phi_{ij}(\mathcal{M})S_{j}^{\Phi}(\mathcal{H})\right]$. Now we can define $\pi^{\Phi}:\mathcal{M}\rightarrow \mathscr{L}(\mathcal{H}^{\Phi},\mathcal{K}^{\Phi})$ as follows:
\begin{align*}
&\hspace{-1cm}\pi^{\Phi}(x)\Big(\sum_{s=1}^{m}\pi^{\varphi}(a_{1s})S_{1}^{\Phi}\xi_{1s},\dots,\sum_{s=1}^{m}\pi^{\varphi}(a_{ns})S_{n}^{\Phi}\xi_{ns}\Big)
:=\\
&=\Big(\sum_{j=1}^{n}\sum_{s=1}^{m}\Phi_{1j}(xa_{js})\xi_{js},\dots,\sum_{j=1}^{n}\sum_{s=1}^{m}\Phi_{nj}(xa_{js})\xi_{js}\Big)\,,
\end{align*}
where $x\in \mathcal{M}$, $a_{is}\in \mathscr{A}$, $\xi_{is}\in \mathcal{H}$, $1\leq s\leq m$, $m\in\Bbb{N}$.
We claim that $\pi^{\Phi}(x)$ is well defined. Indeed, we have
\begin{align*}
&\hspace{-1cm}\Big|\Big|\pi^{\Phi}(x)\Big(\sum_{s=1}^{m}\pi^{\varphi}(a_{1s})S_{1}^{\Phi}\xi_{1s},
\dots,\sum_{s=1}^{m}\pi^{\varphi}(a_{ns})S_{n}^{\Phi}\xi_{ns}\Big)\Big|\Big|_{\alpha}^{2}=\\
&\Big|\Big|\Big(\sum_{j=1}^{n}\sum_{s=1}^{m}\Phi_{1j}(xa_{js})\xi_{js},\dots,\sum_{j=1}^{n}\sum_{s=1}^{m}\Phi_{nj}(xa_{js})\xi_{js}\Big)\Big|\Big|_{\alpha}^{2}\\
&= \sum_{s,r=1}^{m}\sum_{i,j=1}^{n}\sum_{l=1}^{n}\langle\xi_{is},\Phi_{li}(xa_{is})^{*}\Phi_{lj}(xa_{jr})\xi_{jr}\rangle_{\alpha}\\
&=\sum_{s,r=1}^{m}\sum_{i,j=1}^{n}\langle\xi_{is},\varphi_{ij}(\langle xa_{is},xa_{jr}\rangle)\xi_{jr}\rangle_{\alpha}\\
&=\sum_{s,r=1}^{m}\sum_{i,j=1}^{n}\langle\xi_{is},(S_{i}^{\Phi})^{*}\pi^{\varphi}(a_{is}^{*}\langle x,x\rangle a_{jr})S_{j}^{\Phi}\xi_{jr}\rangle_{\alpha}\\
&=\sum_{s,r=1}^{m}\sum_{i,j=1}^{n}\langle\pi^{\varphi}(a_{is})S_{i}^{\Phi}(\xi_{is}),\pi^{\varphi}(\langle x,x\rangle)\pi^{\varphi}(a_{jr})S_{j}^{\Phi}\xi_{jr}\rangle_{\alpha}=\\
&=\Big\langle\sum_{s=1}^{m}\sum_{i=1}^{n}\pi^{\varphi}(a_{is})S_{i}^{\Phi}(\xi_{is}),\pi^{\varphi}(\langle x,x\rangle)\Big(\sum_{r=1}^{m}\sum_{j=1}^{n}\pi^{\varphi}(a_{jr})S_{j}^{\Phi}\xi_{jr}\Big)\Big\rangle_{\alpha}\\
&\leq\Big|\Big|\pi^{\varphi}(\langle x,x\rangle)\Big|\Big|_{\alpha}\,\Big|\Big|\big(\sum_{r=1}^{m}\sum_{i=1}^{n}\pi^{\varphi}(a_{i,r})S_{i}^{\Phi}\xi_{i,r})\Big|\Big|_{\alpha}^{2}\\
&\leq||x||_{\alpha}^{2}\Big|\Big|\big(\sum_{r=1}^{m}\sum_{i=1}^{n}\pi^{\varphi}(a_{i,r})S_{i}^{\Phi}\xi_{i,r})\Big|\Big|_{\alpha}^{2}
\end{align*}
for all $\alpha\in\Delta$ and therefore $\pi^{\Phi}(x)$ is well defined and continuous. Hence it can be extended to the whole of $\mathcal{K}^{\Phi}$. Now we prove that $\pi^{\Phi}$ is a representation. For showing this, let $x,y\in \mathcal{M}$; $a_{is},b_{jr}\in \mathscr{A}$; $\xi_{is},\eta_{jr}\in \mathcal{H}$; $1\leq i,j\leq n$; $1\leq s\leq l$, $1\leq r\leq m$; $n,m\in \Bbb{N}$. Then we have
\begin{align*}
&\hspace{-1cm}
\Big\langle(\pi^{\Phi}(x))^{*}\pi^{\Phi}(y)\Big(\sum_{r=1}^{m}\sum_{j=1}^{n}\pi^{\varphi}(b_{j,r})S_{j}^{\Phi}\eta_{j,r}\Big),
\sum_{s=1}^{l}\sum_{i=1}^{n}\pi^{\varphi}(a_{i,s})S_{i}^{\Phi}\xi_{i,s}\Big\rangle_{\alpha}\\
&=\Big\langle\sum_{r=1}^{m}\sum_{j=1}^{n}\Phi_{j}(yb_{jr})\eta_{jr},\sum_{s=1}^{l}\sum_{i=1}^{n}\Phi_{i}(xa_{is})\xi_{is}\Big\rangle_{\alpha}\\
&=\sum_{s=1}^{l}\sum_{r=1}^{m}\sum_{i,j=1}^{n}\langle\Phi_{i}(xa_{is})^{*}\Phi_{j}(yb_{jr})\eta_{jr},\xi_{is}\rangle_{\alpha}\\
&=\sum_{s=1}^{l}\sum_{r=1}^{m}\sum_{i,j=1}^{n}\langle\varphi_{ij}(\langle xa_{is},yb_{jr}\rangle)\eta_{jr},\xi_{is}\rangle_{\alpha}\\
&=\sum_{s=1}^{l}\sum_{r=1}^{m}\sum_{i,j=1}^{n}\langle (S_{i}^{\Phi})^{*}\pi(a_{is}^{*}\langle x,y\rangle a_{jr})S_{j}^{\Phi}\eta_{jr},\xi_{is}\rangle_{\alpha}\\
&=\Big\langle\pi(\langle x,y\rangle)\Big(\sum_{r=1}^{m}\sum_{j=1}^{n}\pi(b_{j,r})S_{j}^{\Phi}\eta_{j,r}\Big),\sum_{s=1}^{l}
\sum_{i=1}^{n}\pi(a_{i,s})S_{i}^{\Phi}\xi_{i,s}\Big\rangle_{\alpha}
\end{align*}
for all $\alpha\in\Delta$.
Thus $(\pi^{\Phi}(x))^{*}\pi^{\Phi}(y)=\pi^{\varphi}(\langle x,y\rangle)$ on a dense set and hence they are equal on $\mathcal{H}^{\Phi}$.
Let $W_{i},\,i\in \{1,\dots,n\}$ be the orthogonal projection from $\mathcal{K}$ to $\mathcal{K}_{i}^{\Phi}$. Then $W_{i}^{*}:K_{i}^{\Phi}\rightarrow \mathcal{K}$ is an inclusion map. Hence $W_{i}W_{i}^{*}=I_{\mathcal{K}_{i}^{\Phi}}$ for every $i\in\{1,\dots,n\}$. Now we give a representation for $[\Phi]$. For every $x\in \mathcal{M}$ and $\xi\in \mathcal{H}$, we have
$$
\Phi_{ij}(x)(\xi)=(W_{i}^{\Phi})^{*}\pi^{\Phi}(x)S_{j}^{\Phi}(\xi)\,\, \text{for every} \,\,i,j\in\{1,\ldots,n\}.
$$
\end{proof}

\begin{definition}\label{def:minimal}
Let $[\varphi]$ and $[\Phi]$ be as an Theorem~\ref{KGNS}. We say that a data $(\pi,H,K,S_{1},\dots,S_{n},W_{1},\dots,W_{n})$, is a {\it Stinespring representation} of $([\varphi],[\Phi])$ if the conditions $(1)-(2)$ of Theorem~\ref{KGNS} are satisfied. By $K_{i}$, $i\in\{1,\dots,n\}$ we denote $[\pi(\mathcal{M})S_{i}(\mathcal{H})]$.
Such a representation is said to be {\it minimal} if
\begin{enumerate}
 \item[1)] $H=\left[\bigcup_{i=1}^{n}\pi^{\varphi}(A)S_{i}(\mathcal{H})\right]$;
 \item[2)] $K=\left[\bigcup_{i=1}^{n}\pi(\mathcal{M})S_{i}(\mathcal{H})\right]$.
\end{enumerate}
\end{definition}

\begin{thm}\label{KGNS-1}
Let $[\varphi]$ and $[\Phi]$ be as an Theorem~\ref{KGNS}. Assume that
$$
(\pi,H,K,S_{1},\dots,S_{n},W_{1},\dots,W_{n})\,\, \text{and}\,\, (\pi',H',K',S'_{1},\dots,S'_{n},W'_{1},\dots,W'_{n})
$$
are minimal Stinespring representations, which associated with the continuous $*$-homomorphisms $\pi^{\varphi}:\mathscr{A}\to \mathscr{L}(H)$ ($(\pi^{\varphi})':\mathscr{A}\to \mathscr{L}(H')$), respectively. Then there exist unitary operators $U_{1}:H\rightarrow H'$, $U_{2}:K\rightarrow K'$ such that
\begin{enumerate}
 \item[(1)] $U_{1}S_{i}=S'_{i},\,\forall i\in\{1,\dots,n\}$; $U_{1}\pi^{\varphi}(a)=(\pi^{\varphi})'(a)U_{1}$, $\forall a\in \mathscr{A}$.
 \item[(2)] $U_{2}W_{i}=W'_{i};\,\forall i\in\{1,\dots,n\}$; $U_{2}\pi(x)=\pi'(x)U_{1}$; $\forall x\in \mathcal{M}$.
\end{enumerate}
That is, the following diagram commutes, for all $a\in \mathscr{A}$, $x\in \mathcal{M}$, $i\in\{1,\dots,n\}$
$$
\begin{CD}
\mathcal{H}@>S_{i}>> H@>\pi^{\varphi}(a)>> H@>\pi(x)>>K@<W_{i}<< \mathcal{K}\\
@VV\text{Id}V @VVU_{1}V @VVU_{1}V @VVU_{2}V @VV\text{Id}V \\
\mathcal{H}@>S'_{i}>> H'@>(\pi^{\varphi})'(a)>>H'@>\pi'(x)>>K'@<W'_{i}<< \mathcal{K}
\end{CD}
$$
\end{thm}
\begin{proof}
Let us prove the existence of the unitary map $U_{1}:H\rightarrow H'$. First define $U_{1}$ on the dense subspace of $H$ spanned by $\bigcup_{i=1}^{n}\pi^{\varphi}(A)S_{i}(\mathcal{H})$.
$$
U_{1}\Big(\sum_{s=1}^{m}\sum_{i=1}^{n}\pi^{\varphi}(a_{is})S_{i}(\xi_{is})\Big):=
\sum_{s=1}^{m}\sum_{i=1}^{n}(\pi^{\varphi})'(a_{is})S'_{i}(\xi_{is})\,,
$$
where $a_{is}\in \mathscr{A}$, $\xi_{is}\in \mathcal{H}$, $m\in\Bbb{N}$. It is not difficult to check that $U_{1}$ is a surjective continuous operator. Denote the extension of $U_{1}$ to $H$ by $U_{1}$ itself. Then $U_{1}$ is unitary and satisfies the condition in $(1)$. Now define $U_{2}$ on the dense subspace spanned by $\bigcup_{i=1}^{n}\pi(\mathcal{M})S_{i}(\mathcal{H})$.
\begin{align*}
&\hspace{-1cm} U_{2}\Big(\sum_{s=1}^{m}\pi(x_{1s})S_{1}\xi_{1s}+\dots+\sum_{s=1}^{m}\pi(x_{ns})S_{n}\xi_{ns}\Big)\\
&:=\sum_{s=1}^{m}\pi'(x_{1s})S'_{1}\xi_{1s}+\dots+\sum_{s=1}^{m}\pi'(x_{ns})S'_{n}\xi_{ns},
\end{align*}
where $x_{is}\in \mathcal{M}$, $\xi_{is}\in \mathcal{H}$, $m\in\Bbb{N}$. Using the fact that $S_{i},S_{i}'$ are continuous embeddings
for every $i\in\{1,\dots,n\}$ we have
$$
U_{2}\Big(\sum_{s=1}^{m}\pi(x_{is})S_{i}\xi_{is}\Big)=
\sum_{s=1}^{m}\pi'(x_{is})S'_{i}\xi_{is},
$$
and so $U_{2}(K_{i})=K'_{i}$. We can see that $U_{2}$ is well defined and can be extended to a unitary map. For showing this, consider

\begin{align*}
&\hspace{-2cm} \Big|\Big| \sum_{s=1}^{m}\pi'(x_{1s})S'_{1}\xi_{1s}+\dots+\sum_{s=1}^{m}\pi'(x_{ns})S'_{n}\xi_{ns} \Big|\Big|_{\alpha}^{2}\\
&=\Big\langle\sum_{i=1}^{n}\sum_{s=1}^{m}\pi'(x_{is})S'_{i}\xi_{is},
\sum_{j=1}^{n}\sum_{r=1}^{m}\pi'(x_{jr})S'_{j}\xi_{jr}\Big\rangle_{\alpha}\\
&=\sum_{s,r=1}^{m}\sum_{i,j=1}^{n}\langle\pi'(x_{is})S'_{i}\xi_{is},
\pi'(x_{jr})S'_{j}\xi_{jr}\rangle_{\alpha}\\
&=\sum_{s,r=1}^{m}\sum_{i,j=1}^{n}\langle\xi_{is},
S'^{*}_{i}(\pi^{\varphi})'(\langle x_{is},x_{jr}\rangle)S'_{j}(\xi_{jr})\rangle_{\alpha}\\
&=\sum_{s,r=1}^{m}\sum_{i,j=1}^{n}\langle\xi_{is},
\varphi_{ij}(\langle xa_{is},xa_{jr}\rangle)(\xi_{jr})\rangle_{\alpha}\\
&=\sum_{s,r=1}^{m}\sum_{i,j=1}^{n}\langle\xi_{is},
S^{*}_{i}\pi^{\varphi}(\langle x_{is},x_{jr}\rangle)S_{j}(\xi_{jr})\rangle_{\alpha}\\
&=\sum_{s,r=1}^{m}\sum_{i,j=1}^{n}\langle\pi(x_{is})S_{i}\xi_{is},
\pi(x_{jr})S_{j}\xi_{jr}\rangle_{\alpha}\\
&=\Big\langle\sum_{i=1}^{n}\sum_{s=1}^{m}\pi(x_{is})S_{i}\xi_{is},
\sum_{j=1}^{n}\sum_{r=1}^{m}\pi(x_{jr})S_{j}\xi_{jr}\Big\rangle_{\alpha}\\
&=\Big|\Big|
\sum_{s=1}^{m}\pi(x_{1s})S_{1}\xi_{1s}+\dots+\sum_{s=1}^{m}\pi(x_{ns})S_{n}\xi_{ns}\Big|\Big|_{\alpha}^{2}
\end{align*}
for all $\alpha\in\Delta$. Hence $U_{2}$ is well defined and continuous, therefore $U_{2}$ can be extended to whole of $K$. We denote this extension by $U_{2}$ itself. Evidently
$U_{2}$ is a surjective continuous operator. We notice that
$$
(\pi,H,K,S_{1},\dots,S_{n},W_{1},\dots,W_{n})\,\, \text{and}\,\,(\pi',H',K',S'_{1},\dots,S'_{n},W'_{1},\dots,W'_{n})
$$
are Stinespring representations for $([\varphi], [\Phi])$. Hence for every $i\in\{1,\dots,n\}$ we have
$$
\Phi_{ij}(x)=W_{i}^{*}\pi(x)S_{j}=W'^{*}_{i}\pi'(x)S'_{j}=W'^{*}_{i}U_{2}\pi(x)S_{j}.
$$
Hence
$$
(W_{i}^{*}-W'^{*}_{i}U_{2})\Psi(x)S_{j}=0,
$$
whence
$$
(W_{i}^{*}-W'^{*}_{i}U_{2})\pi(x)S_{j}(\xi)=0
$$
for all $x\in \mathcal{M},\,\xi\in \mathcal{H},\,\, i,j\in \{1,\dots,n\}$.\\
Hence $U_{2}W_{i}=W'_{i}$ for every $i\in\{1,\dots,n\}$. Finally, we show that
$U_{2}\pi(x)=\pi'(x)U_{1}$ on the dense subspace
$$
\Big\{\sum_{s=1}^{m}\sum_{i=1}^{n}\pi(a_{is})S_{i}(\xi_{is});\,a_{is}\in \mathscr{A},\,\xi_{is}\in \mathcal{H},\,m\in\Bbb{N}\Big\}.
$$
We must recall that every representation $\pi:\mathcal{M}\rightarrow L(H,K)$ has the property $\pi(xa)=\pi(x)\pi(a)$ for every $x\in \mathcal{M}$ and $a\in\mathscr{A}$. Utilizing the fact that $\pi$ and $\pi'$ are representations associated with $\pi^{\varphi}$ and $(\pi^{\varphi})'$, respectively, we have
\begin{align*}
U_{2}\pi(x)\Big(\sum_{s=1}^{m}\sum_{i=1}^{n}\pi^{\varphi}(a_{is})S_{i}(\xi_{is})\Big)&=U_{2}\Big(\sum_{s=1}^{m}\sum_{i=1}^{n}\pi(xa_{is})S_{i}\xi_{is}\Big)\\
&=\sum_{s=1}^{m}\sum_{i=1}^{n}\pi'(xa_{is})S'_{i}\xi_{is}\\
&=\pi'(x)\Big(\sum_{s=1}^{m}\sum_{i=1}^{n}(\pi^{\varphi})'(a_{is})S'_{i}(\xi_{is})\Big)\\
&=\pi'(x)U_{1}\Big(\sum_{s=1}^{m}\sum_{i=1}^{n}\pi^{\varphi}(a_{is})S_{i}(\xi_{is})\Big)\,.
\end{align*}
\end{proof}

\section{Radon--Nikodym type theorem}

Consider a full Hilbert $\mathscr{A}$-module $\mathcal{M}$ over a locally $C^*$-algebra $\mathscr{A}$ and locally Hilbert spaces $\mathcal{H},\mathcal{K}$.
Let $[\Phi],[\Psi]\in\mathcal{CP}_{n}(\mathcal{M},\mathscr{L}(\mathcal{H},\mathcal{K}))$. We say that $[\Phi]=(\Phi_{ij})_{i,j=1}^{n}$ is equivalent to   $[\Psi]=(\Psi_{ij})_{i,j=1}^{n}$, denoted by $[\Phi]\sim[\Psi]$, if
\begin{gather*}
\langle[\Phi](x),[\Phi](x)\rangle=\langle[\Psi](x),[\Psi](x)\rangle
\end{gather*}
for every $x\in \mathcal{M}$. We remark that the relation $\sim$ is an equivalence relation on $\mathcal{CP}_{n}(\mathcal{M},\mathscr{L}(\mathcal{H},\mathcal{K}))$.

\begin{lemma}\label{Eq}
Let $\mathscr{A}$ be a unital locally $C^*$-algebra and $\mathcal{M}$ be a full Hilbert $\mathscr{A}$-module over $\mathscr{A}$. Let $\mathcal{H},\mathcal{K}$ be locally Hilbert spaces and $[\Phi],[\Psi]\in\mathcal{CP}_{n}(\mathcal{M},\mathscr{L}(\mathcal{H},\mathcal{K}))$. Then $\Psi\sim\Phi$ if and only if the Stinespring constructions associated
with $\Psi$ and $\Phi$ are unitarily equivalent.
\end{lemma}
\begin{proof}
Let $\Psi\sim\Phi$. We must prove that
the Stinespring constructions of $[\Phi]$ and $[\Psi]$ are unitary equivalent.
For every $x\in\mathcal{M}$, we have
\begin{gather*}
(\Phi_{ij}(x)_{i,j=1}^{n})^{*}(\Phi_{ij}(x))_{i,j=1}^{n}=
(\varphi_{ij}\langle x,x\rangle)_{i,j=1}^{n}=(\Psi_{ij}(x)_{i,j=1}^{n})^{*}(\Psi_{ij}(y))_{i,j=1}^{n}
\end{gather*}
for some continuous completely $n$-positive map $(\varphi_{ij})$. By \cite[Theorem~4.1.8]{J-0}, there exists a unitary operator $U_{1}\in \mathscr{L}(\mathcal{H}^{\Phi},\mathcal{H}^{\Psi})$ such that $U_{1}S_{i}^{\Phi}=S_{i}^{\Psi}$ for every $i\in\{1,\dots,n\}$.
Observe that by the Theorem~\ref{KGNS} the elements
\begin{gather*}
\Big(\sum\limits_{i=1}^{n}\sum\limits_{s=1}^{m}\Phi_{1i}(x_{is})\xi_{is},\dots,\sum\limits_{i=1}^{n}\sum\limits_{s=1}^{m}\Phi_{ni}(x_{is})\xi_{is}\Big);\\
x_{is}\in\mathcal{M};\, \xi_{is}\in \mathcal{H};\, 1\leq i\leq n;\, 1\leq s\leq m;\, m\in\Bbb{N},
\end{gather*}
are dense in the locally Hilbert space $\mathcal{K}^{\Phi}$. Now for all $\alpha\in\Delta$ we have
\begin{eqnarray*}
&&\hspace{-3cm}\Big|\Big|\Big(\sum\limits_{i=1}^{n}\sum\limits_{s=1}^{m}
\Phi_{1i}(x_{is})\xi_{is},\dots,\sum\limits_{i=1}^{n}\sum\limits_{s=1}^{m}\Phi_{ni}(x_{is})\xi_{is}\Big)\Big|\Big|_{\alpha}^{2}\\
&=&\sum_{s,r=1}^{m}\sum_{i,j=1}^{n}\sum_{l=1}^{n}\langle\xi_{is},\Phi_{li}(xa_{is})^{*}\Phi_{lj}(xa_{jr})\xi_{jr}\rangle_{\alpha}\\
&=&\sum_{s,r=1}^{m}\sum_{i,j=1}^{n}\langle\xi_{is},\varphi_{ij}(\langle x_{is},x_{jr}\rangle)\xi_{jr}\rangle_{\alpha}\\
&=&\sum_{s,r=1}^{m}\sum_{i,j=1}^{n}\sum_{l=1}^{n}\langle\xi_{is},\Psi_{li}(xa_{is})^{*}\Psi_{lj}(xa_{jr})\xi_{jr}\rangle_{\alpha}\\
&=&\Big|\Big|\Big(\sum\limits_{i=1}^{n}\sum\limits_{s=1}^{m}
\Psi_{1i}(x_{is})\xi_{is},\dots,\sum\limits_{i=1}^{n}\sum\limits_{s=1}^{m}\Psi_{ni}(x_{is})\xi_{is}\Big)\Big|\Big|_{\alpha}^{2}
\end{eqnarray*}
Since the elements $\Big(\sum\limits_{i=1}^{n}\sum\limits_{s=1}^{m}\Psi_{1i}(x_{is})\xi_{is},\dots,\sum\limits_{i=1}^{n}\sum\limits_{s=1}^{m}\Psi_{ni}(x_{is})\xi_{is}\Big)$ are dense in the locally Hilbert space
$\mathcal{K}^{\Psi}$, there exists a continuous linear operator $U_{2}:\mathcal{K}^{\Phi}\to \mathcal{K}^{\Psi}$ defined by
\begin{gather*}
\hspace{-3cm}U_{2}\Big(\sum\limits_{i=1}^{n}\sum\limits_{s=1}^{m}\Phi_{1i}(x_{is})\xi_{is},\dots,\sum\limits_{i=1}^{n}\sum\limits_{s=1}^{m}\Phi_{ni}(x_{is})\xi_{is}\Big)\\
=\Big(\sum\limits_{i=1}^{n}\sum\limits_{s=1}^{m}\Psi_{1i}(x_{is})\xi_{is},\dots,\sum\limits_{i=1}^{n}\sum\limits_{s=1}^{m}\Psi_{ni}(x_{is})\xi_{is}\Big)
\end{gather*}
on a dense subspace of $\mathcal{K}^{\Phi}$. Hence it can be extended to the whole space $\mathcal{K}^{\Phi}$ and
$U_{2}W^{\Phi}=W^{\Psi}$.
Let us prove that for every $x\in\mathcal{M}$ the equality $U_{2}\pi^{\Phi}(x)=\pi^{\Psi}(x)U_{1}$ holds.
To this end, take $x\in\mathcal{M}$. We have
\begin{eqnarray*}
&&\hspace{-1in}U_{2}\pi^{\Phi}(x)\Big(\sum_{s=1}^{m}\pi^{\varphi}(a_{1s})S_{1}^{\Phi}\xi_{1s},\dots,\sum_{s=1}^{m}\pi^{\varphi}(a_{ns})S_{n}^{\Phi}\xi_{ns}\Big)\\
&=&U_{2}\Big(\sum_{i=1}^{n}\sum_{s=1}^{m}\Phi_{1i}(xa_{1s})\xi_{1s},\dots,\sum_{i=1}^{n}\sum_{s=1}^{m}\Phi_{ni}(xa_{ns})\xi_{ns}\Big)\\
&=&\Big(\sum_{i=1}^{n}\sum_{s=1}^{m}\Psi_{1i}(xa_{1s})\xi_{1s},\dots,\sum_{i=1}^{n}\sum_{s=1}^{m}\Psi_{ni}(xa_{ns})\xi_{ns}\Big)\\
&=&\pi^{\Psi}(x)\Big(\sum_{s=1}^{m}\pi^{\psi}(a_{1s})S_{1}^{\Psi}\xi_{1s},\dots,\sum_{s=1}^{m}\pi^{\psi}(a_{ns})S_{n}^{\Psi}\xi_{ns}\Big)\\
&=&\pi^{\Psi}(x)U_{1}\Big(\sum_{s=1}^{m}\pi^{\varphi}(a_{1s})S_{1}^{\Phi}\xi_{1s},\dots,\sum_{s=1}^{m}\pi^{\varphi}(a_{ns})S_{n}^{\Phi}\xi_{ns}\Big).
\end{eqnarray*}
Since linear continuous operators $U_{2}\pi^{\Phi}(x)$ and $\pi^{\Psi}(x)U_{1}$ coincide on a dense subspace of the space $\mathcal{H}^{\Phi}$, they coincide on the whole space. Hence the unitarily equivalence of the Stinespring constructions is established.

On the other hand, assume that the Stinespring constructions of $[\Phi]$ and $[\Psi]$ are unitarily equivalent.
Take $x\in\mathcal{M}$, $i,j\in\{1,\dots,n\}$. Then we may write
\begin{eqnarray*}
\psi_{ij}\langle x,x\rangle &=&\sum_{k=1}^{n}\Psi_{ki}(x)^{*}\Psi_{kj}(x)\\
&=&\sum_{k=1}^{n}\Big((W_{k}^{\Psi})^*\pi^{\Psi}(x)S_{i}^{\Psi}\Big)^*(W_{k}^{\Psi})^*\pi^{\Psi}(x)S_{j}^{\Psi}\\
&=&\sum_{k=1}^{n}\Big((U_{2}W_{k}^{\Phi})^*\pi^{\Psi}(x)U_{1}S_{i}^{\Phi}\Big)^*(U_{2}W_{k}^{\Phi})^*\pi^{\Psi}(x)U_{1}S_{j}^{\Phi}\\
&=&\sum_{k=1}^{n}\Big((W_{k}^{\Phi})^*U_{2}^*U_{2}\pi^{\Phi}(x)S_{i}^{\Phi}\Big)^*(W_{k}^{\Phi})^*U_{2}^*U_{2}\pi^{\Phi}(x)S_{j}^{\Phi}\\
&=&\sum_{k=1}^{n}\Big((W_{k}^{\Phi})^*\pi^{\Phi}(x)S_{i}^{\Phi}\Big)^*(W_{k}^{\Phi})^*\pi^{\Phi}(x)S_{j}^{\Phi}\\
&=&\sum_{k=1}^{n}\Psi_{ki}(x)^{*}\Psi_{kj}(x)=\varphi_{ij}\langle x,x\rangle.
\end{eqnarray*}
Thus $[\Psi]\sim[\Phi]$.
\end{proof}

Let $[\Phi],[\Psi]\in\mathcal{CP}_{n}(\mathcal{M},\mathscr{L}(\mathcal{H},\mathcal{K}))$. We say that $[\Phi]$ is dominated by $[\Psi]$ and denoted by $[\Phi]\preceq[\Psi]$,
if $\langle[\Psi](x),[\Psi](x)\rangle\leq\langle[\Phi](x),[\Phi](x)\rangle$ for every $x\in\mathcal{M}$.
The following properties of the relation $\preceq$ are evident:
\begin{itemize}
\item~$\Phi\preceq\Phi$ for every $\Phi\in\mathcal{CP}_{n}(\mathcal{M},\mathscr{L}(\mathcal{H},\mathcal{K}))$;
\item~ If $\Phi_{1}\preceq\Phi_{2}$, $\Phi_{2}\preceq\Phi_{3}$ then $\Phi_{2}\preceq\Phi_{3}$ for every $\Phi_{1},\Phi_{2},\Phi_{3}\in\mathcal{CP}_{n}(\mathcal{M},\mathscr{L}(\mathcal{H},\mathcal{K}))$;
\item~If $\Phi\preceq\Psi$ and $\Psi\preceq\Phi$ then $\Psi\thicksim\Phi$ for every $\Phi,\Psi\in\mathcal{CP}_{n}(\mathcal{M},\mathscr{L}(\mathcal{H},\mathcal{K}))$.
\end{itemize}

Let $\mathcal{H}_{1},\mathcal{H}_{2}$ be locally Hilbert spaces with the same some index set $\Delta$. By $\mathcal{H}_{1}\oplus\mathcal{H}_{2}$ we denote
the direct sum topological vector spaces $\mathcal{H}_{1}$ and $\mathcal{H}_{2}$.

Let $\Pi:\mathcal{M}\to \mathscr{L}(\mathcal{H}_{1},\mathcal{H}_{2})$ be a representation of $\mathcal{M}$ on locally Hilbert spaces $\mathcal{H}_{1}$ and $\mathcal{H}_{2}$. The set
$$
\Pi(\mathcal{M})':=\{T\oplus N\in \mathscr{L}(\mathcal{H}_{1}\oplus \mathcal{H}_{2}):\,\Pi(x)T=N\Pi(x);\,\Pi(x)^*N=T\Pi(x)^*;\, x\in \mathcal{M}\}
$$
is called the {\it commutant} of $\Pi(\mathcal{M})$.

\begin{lemma}\label{com}
Let $[\Phi]\in\mathcal{CP}_{n}(\mathcal{M},\mathscr{L}(\mathcal{H},\mathcal{K}))$ and $(\pi^{\Phi},\mathcal{H}^{\Phi},\mathcal{K}^{\Phi},S_{1}^{\Phi},\dots,S_{n}^{\Phi},W_{1}^{\Phi},\dots,W_{n}^{\Phi})$ be the Stinespring construction
associated with $[\Phi]$. If the operator $T\oplus N\in\pi^{\Phi}(\mathcal{M})'$ is positive, then
there exists $[\Phi^{T\oplus N}]\in\mathcal{CP}_{n}(\mathcal{M},\mathscr{L}(\mathcal{H},\mathcal{K}))$, defined by the formula
$$
\Phi_{ij}^{T\oplus N}(x)=(W_{i}^{\Phi})^*\sqrt{N}\pi^{\Phi}(x)\sqrt{T}S_{j}^{\Phi}
$$
\end{lemma}
\begin{proof}
For every $x,y\in \mathcal{M}$ we may write
\begin{align*}
\sum_{k=1}^{n}(\Phi_{ki}^{T\oplus N})(x)^*(\Phi_{kj}^{T\oplus N})(y)&=\sum_{k=1}^{n}(S_{i}^{\Phi})^*\sqrt{T}\pi^{\Phi}(x)^*\sqrt{N}W_{k}^{\Phi}
(W_{k}^{\Phi})^*\sqrt{N}\pi^{\Phi}(y)\sqrt{T}S_{j}^{\Phi}\\
&=(S_{i}^{\Phi})^*T\pi^{\Phi}(x)^*\Big(\sum_{k=1}^{n}W_{k}^{\Phi}(W_{k}^{\Phi})^*\Big)\pi^{\Phi}(y)TS_{j}^{\Phi}\\
&=(S_{i}^{\Phi})^*T\pi^{\Phi}(x)^*\pi^{\Phi}(y)TS_{j}^{\Phi}\\
&=(S_{i}^{\Phi})^*T^{2}\pi^{\varphi}(\langle x,y\rangle)S_{j}^{\Phi}\\
&=\varphi_{ijT^{2}}(\langle x,y\rangle).
\end{align*}
Employing \cite[Lemma~4.2.2]{J-0} we deduce that $[\varphi_{T^2}]=(\varphi_{ijT^2})_{ij}^{n}$ is continuous completely $n$-positive map from
$\mathscr{A}$ to $\mathscr{L}(\mathcal{H})$ and therefore $[\Phi^{T\oplus N}]\in\mathcal{CP}_{n}(\mathcal{M},\mathscr{L}(\mathcal{H},\mathcal{K}))$.
\end{proof}

\begin{lemma}\label{commut}
Let $\Pi:\mathcal{M}\to \mathscr{L}(\mathcal{H},\mathcal{K})$ be a representation of a full Hilbert $\mathscr{A}$-module $\mathcal{M}$ on locally Hilbert spaces $\mathcal{H}$ and $\mathcal{K}$.
Then $\Pi(\mathcal{M})'$ is a locally $C^*$-algebra and $T\in\pi(\mathscr{A})'$. Moreover, if the representation $\Pi$ is nondegenerate, then  for given $T\oplus N\in\Pi(\mathcal{M})'$ the operator $N$ is uniquely determined by $T$.
\end{lemma}
\begin{proof}
Let $T_{1}\oplus N_{1},T_{2}\oplus N_{2}\in\Pi(\mathcal{M})'$ and $\zeta\in\Bbb{C}$. We have
\begin{gather*}
T_{1}\oplus N_{1}+T_{2}\oplus N_{2}=(T_{1}+T_{2})\oplus (N_{1}+N_{2})\in\Pi(\mathcal{M})';\\
\zeta(T_{1}\oplus N_{1})=\zeta T_{1}\oplus \zeta N_{1}\in\Pi(\mathcal{M})';\,
T_{1}^{*}\oplus N_{1}^{*}\in\Pi(\mathcal{M})';\\
(T_{1}\circ T_{2})\oplus (N_{1}\circ N_{2})\in\Pi(\mathcal{M})'.
\end{gather*}
Clearly $\Pi(\mathcal{M})'$ is a closed subalgebra of the locally $C^{*}$-algebra $\mathscr{L}(\mathcal{H}\oplus\mathcal{K})$ and
therefore $\Pi(\mathcal{M})'$ is a locally $C^{*}$-algebra. Take $T\oplus N\in\Pi(\mathcal{M})'$. Then we have
\begin{gather*}
\pi(\langle x,y\rangle)T=\langle \Pi(x)\Pi(y)\rangle T=\Pi(x)^{*}\Pi(y)T=\\
\Pi(x)^{*}N\Pi(y)=T \Pi(x)^{*}\Pi(y)=T\langle \Pi(x)\Pi(y)\rangle=T\pi(\langle x,y\rangle)
\end{gather*}
Taking into account that the linear span of the set $\{\langle x,y\rangle:\, x,y\in \mathcal{M}\}$ is dense in $\mathscr{A}$ we deduce
that $T\in\pi(\mathscr{A})'$.
Assume the representation $\Pi:\mathcal{M}\to \mathscr{L}(\mathcal{H},\mathcal{K})$ is nondegenerate and
$T\oplus N\in\Pi(\mathcal{M})'$.
Since the subspace $[\Pi(\mathcal{M})(\mathcal{H})]$ is dense in $\mathcal{K}$ for given $T$ the equality
$\Pi(x)T(\xi)=N(\Pi(x)\xi)$, $\xi\in\mathcal{H}$, $x\in \mathcal{M}$ determines the operator $N$.
\end{proof}

The following noncommutative version of the Radon--Nikodym theorem is the main result of this section.

\begin{thm}\label{RN}
Let $[\Phi],[\Psi]\in\mathcal{CP}_{n}(\mathcal{M},\mathscr{L}(\mathcal{H},\mathcal{K}))$ and $[\Psi]\preceq[\Phi]$. Then there exists a unique positive linear operator $\Delta_{\Psi}^{\Phi}\in(\pi^{\Phi}(\mathcal{M})')$ such that $[\Psi]\thicksim [\Phi^{\sqrt{\Delta_{\Phi}^{\Psi}}}]$.
\end{thm}
\begin{proof}
Take the Stinespring constructions $(\pi^{\Phi},\mathcal{H}^{\Phi},\mathcal{K}^{\Phi},S_{1}^{\Phi},\dots,S_{n}^{\Phi},W_{1}^{\Phi},\dots,W_{n}^{\Phi})$ and also $(\pi^{\Psi},\mathcal{H}^{\Psi},\mathcal{K}^{\Psi},S_{1}^{\Psi},\dots,S_{n}^{\Psi},W_{1}^{\Psi},\dots,W_{n}^{\Psi})$ associated with $[\Phi]$ and $[\Psi]$, respectively. If $[\Psi]\preceq[\Phi]$, then $[\psi]\preceq[\phi]$ and by utilizing \cite[Lemma~4.2.5]{J-0} there exists a linear continuous operator $R:\mathcal{H}^{\Phi}\to \mathcal{H}^{\Psi}$ such that
\begin{align*}
&\hspace{-2cm}R\Big(\sum_{s=1}^{m}\pi^{\varphi}(a_{1s})S_{1}^{\Phi}\xi_{1s},\dots,\sum_{s=1}^{m}\pi^{\varphi}(a_{ns})S_{n}^{\Phi}\xi_{ns}\Big)\\
&=\Big(\sum_{s=1}^{m}\pi^{\psi}(a_{1s})S_{1}^{\Psi}\xi_{1s},\dots,\sum_{s=1}^{m}\pi^{\psi}(a_{ns})S_{n}^{\Psi}\xi_{ns}\Big)\,,
\end{align*}
where $x\in \mathcal{M}$, $a_{is}\in \mathscr{A}$, $\xi_{is}\in \mathcal{H}$, $1\leq i\leq n$, $1\leq s\leq m$, $m\in\Bbb{N}$.
Moreover $\|R_{\alpha}\|\leq 1$ for all $\alpha\in\Delta$ and $\pi^{\psi}(a)=\pi^{\varphi}_{R^*R}(a)$ for every $a\in \mathscr{A}$. We have
\begin{align*}
&\hspace{-2cm}\Big|\Big|\Big(\sum\limits_{i=1}^{n}\sum\limits_{s=1}^{m}\Psi_{1i}(x_{is})\xi_{is},
\dots,\sum\limits_{i=1}^{n}\sum\limits_{s=1}^{m}\Psi_{ni}(x_{is})\xi_{is}\Big)\Big|\Big|_{\alpha}^{2}\\
&=\sum_{s,r=1}^{m}\sum_{i,j=1}^{n}\sum_{l=1}^{n}\langle\xi_{is},\Psi_{li}(xa_{is})^{*}\Psi_{lj}(xa_{jr})\xi_{jr}\rangle_{\alpha}\\
&=\sum_{s,r=1}^{m}\sum_{i,j=1}^{n}\langle\xi_{is},\psi_{ij}(\langle x_{is},x_{jr}\rangle)\xi_{jr}\rangle_{\alpha}\\
&\leq\sum_{s,r=1}^{m}\sum_{i,j=1}^{n}\langle\xi_{is},\phi_{ij}(\langle x_{is},x_{jr}\rangle)\xi_{jr}\rangle_{\alpha}\\
&=\sum_{s,r=1}^{m}\sum_{i,j=1}^{n}\sum_{l=1}^{n}\langle\xi_{is},\Phi_{li}(xa_{is})^{*}\Phi_{lj}(xa_{jr})\xi_{jr}\rangle_{\alpha}\\
&=\Big|\Big|\Big(\sum\limits_{i=1}^{n}\sum\limits_{s=1}^{m}\Phi_{1i}(x_{is})\xi_{is},
\dots,\sum\limits_{i=1}^{n}\sum\limits_{s=1}^{m}\Phi_{ni}(x_{is})\xi_{is}\Big)\Big|\Big|_{\alpha}^{2}.
\end{align*}
Since the elements $\Big(\sum\limits_{i=1}^{n}\sum\limits_{s=1}^{m}\Phi_{1i}(x_{is})\xi_{is},
\dots,\sum\limits_{i=1}^{n}\sum\limits_{s=1}^{m}\Phi_{ni}(x_{is})\xi_{is}\Big)$
are dense in the space $\mathcal{K}^{\Phi}$, we deduce that there exists a unitary operator $Q:\mathcal{K}^{\Phi}\to \mathcal{K}^{\Psi}$ such that
\begin{gather*}
\hspace{-2cm}Q\Big(\sum\limits_{i=1}^{n}\sum\limits_{s=1}^{m}\Phi_{1i}(x_{is})\xi_{is},
\dots,\sum\limits_{i=1}^{n}\sum\limits_{s=1}^{m}\Phi_{ni}(x_{is})\xi_{is}\Big)\\
=\Big(\sum\limits_{i=1}^{n}\sum\limits_{s=1}^{m}\Psi_{1i}(x_{is})\xi_{is},
\dots,\sum\limits_{i=1}^{n}\sum\limits_{s=1}^{m}\Psi_{ni}(x_{is})\xi_{is}\Big)
\end{gather*}
It is clear that $\|Q_{\alpha}
\|\leq 1$ for all $\alpha\in\Delta$.
Now take an arbitrary element $x\in \mathcal{M}$. Then
\begin{align*}
&\hspace{-2cm}Q\pi^{\Phi}(x)\Big(\sum_{s=1}^{m}\pi^{\varphi}(a_{1s})S_{1}^{\Phi}\xi_{1s},\dots,\sum_{s=1}^{m}\pi^{\varphi}(a_{ns})S_{n}^{\Phi}\xi_{ns}\Big)\\
&=Q\Big(\sum\limits_{i=1}^{n}\sum\limits_{s=1}^{m}\Phi_{1i}(x_{is})\xi_{is},
\dots,\sum\limits_{i=1}^{n}\sum\limits_{s=1}^{m}\Phi_{ni}(x_{is})\xi_{is}\Big)\\
&=\Big(\sum\limits_{i=1}^{n}\sum\limits_{s=1}^{m}\Psi_{1i}(x_{is})\xi_{is},
\dots,\sum\limits_{i=1}^{n}\sum\limits_{s=1}^{m}\Psi_{ni}(x_{is})\xi_{is}\Big)\\
&=\pi^{\Psi}(x)\Big(\sum_{s=1}^{m}\pi^{\psi}(a_{1s})S_{1}^{\Psi}\xi_{1s},\dots,\sum_{s=1}^{m}\pi^{\psi}(a_{ns})S_{n}^{\Psi}\xi_{ns}\Big)\\
&=\pi^{\Psi}(x)R\Big(\sum_{s=1}^{m}\pi^{\varphi}(a_{1s})S_{1}^{\Phi}\xi_{1s},\dots,\sum_{s=1}^{m}\pi^{\varphi}(a_{ns})S_{n}^{\Phi}\xi_{ns}\Big)\,,
\end{align*}
where $a_{is}\in \mathscr{A}$, $\xi_{is}\in \mathcal{H}$, $1\leq i\leq n$, $1\leq s\leq m$, $m\in\Bbb{N}$. Since the elements
$$
\Big(\sum\limits_{s=1}^{m}\pi^{\varphi}(a_{1s})S_{1}^{\Phi}\xi_{1s},\dots,\sum\limits_{s=1}^{m}\pi^{\varphi}(a_{ns})S_{n}^{\Phi}\xi_{ns}\Big)
$$
are dense in the locally Hilbert space
$\mathcal{H}^{\Phi}$ we have $Q\pi^{\Phi}(x)=\pi^{\Psi}(x)R$.

Take again $x,y\in \mathcal{M}$. Then
\begin{align*}
&\hspace{-1.5cm}\pi^{\Psi}(x)^*Q\Big(\sum\limits_{i=1}^{n}\sum\limits_{s=1}^{m}\Phi_{1i}(ya_{is})\xi_{is},
\dots,\sum\limits_{i=1}^{n}\sum\limits_{s=1}^{m}\Phi_{ni}(ya_{is})\xi_{is}\Big)\\
&=\pi^{\Psi}(x)^*\Big(\sum\limits_{i=1}^{n}\sum\limits_{s=1}^{m}\Psi_{1i}(ya_{is})\xi_{is},
\dots,\sum\limits_{i=1}^{n}\sum\limits_{s=1}^{m}\Psi_{ni}(ya_{is})\xi_{is}\Big)\\
&=\pi^{\Psi}(x)^*\pi^{\Psi}(y)\Big(\sum_{s=1}^{m}\pi^{\psi}(a_{1s})S_{1}^{\Psi}\xi_{1s},\dots,\sum_{s=1}^{m}\pi^{\psi}(a_{ns})S_{n}^{\Psi}\xi_{ns}\Big)\\
&=\pi^{\psi}(\langle x,y\rangle)\Big(\sum_{s=1}^{m}\pi^{\psi}(a_{1s})S_{1}^{\Psi}\xi_{1s},\dots,\sum_{s=1}^{m}\pi^{\psi}(a_{ns})S_{n}^{\Psi}\xi_{ns}\Big)\\
&=R\pi^{\varphi}(\langle x,y\rangle)\Big(\sum_{s=1}^{m}\pi^{\varphi}(a_{1s})S_{1}^{\Phi}\xi_{1s},\dots,\sum_{s=1}^{m}\pi^{\varphi}(a_{ns})S_{n}^{\Phi}\xi_{ns}\Big)\\
&=R\pi^{\Phi}(x)^*\pi^{\Phi}(y)\Big(\sum_{s=1}^{m}\pi^{\varphi}(a_{1s})S_{1}^{\Phi}\xi_{1s},\dots,\sum_{s=1}^{m}\pi^{\varphi}(a_{ns})S_{n}^{\Phi}\xi_{ns}\Big)\\
&=R\pi^{\Phi}(x)^*\Big(\sum\limits_{i=1}^{n}\sum\limits_{s=1}^{m}\Phi_{1i}(ya_{is})\xi_{is},
\dots,\sum\limits_{i=1}^{n}\sum\limits_{s=1}^{m}\Phi_{ni}(ya_{is})\xi_{is}\Big)\,,
\end{align*}
where $a_{is}\in \mathscr{A}$, $\xi_{is}\in \mathcal{H}$, $1\leq i\leq n$, $1\leq s\leq m$, $m\in\Bbb{N}$. Taking into account that the sums
$$
\Big(\sum\limits_{i=1}^{n}\sum\limits_{s=1}^{m}\Phi_{1i}(ya_{is})\xi_{is},
\dots,\sum\limits_{i=1}^{n}\sum\limits_{s=1}^{m}\Phi_{ni}(ya_{is})\xi_{is}\Big)
$$
are dense in the locally Hilbert space $\mathcal{K}^{\Phi}$ we deduce that $\pi^{\Psi}(x)^*Q=R\pi^{\Phi}(x)^*$.

We define the operator $\Delta_{\Psi}^{\Phi}$ by $\Delta_{\Psi}^{\Phi}=\Delta_{1}\oplus\Delta_{2}$, where $\Delta_{1}=R^*R$ and $\Delta_{2}=Q^*Q$. For every $x\in \mathcal{M}$ the following equalities hold:
\begin{gather*}
\Delta_{2}\pi^{\Phi}(x)=Q^*Q\pi^{\Phi}(x)=Q^*\pi^{\Psi}(x)R
=\pi^{\Phi}(x)R^*R=\pi^{\Phi}(x)\Delta_{1}.
\end{gather*}
The same equalities are true for the adjoint operator as
\begin{gather*}
\pi^{\Phi}(x)^*\Delta_{2}=\pi^{\Phi}(x)^*Q^*Q=R^*\pi^{\Psi}(x)^*Q
=R^*R\pi^{\Phi}(x)^*=\Delta_{1}\pi^{\Phi}(x)^*.
\end{gather*}
Hence $\Delta_{\Psi}^{\Phi}\in(\pi^{\Phi}(\mathcal{M})')$ and $0\leq\Delta_{\Psi}^{\Phi}\leq 1$. The same calculations as in Lemma~\ref{com} for every $x\in \mathcal{M}$ give us equalities
\begin{align*}
\Big(\langle[\Phi^{\sqrt{\Delta_{\Phi}^{\Psi}}}(x)],[\Phi^{\sqrt{\Delta_{\Phi}^{\Psi}}}(y)]\rangle\Big)_{ij}&=
\sum_{r=1}^{n}(\Phi^{\sqrt{\Delta_{\Phi}^{\Psi}}}_{ri}(x))^*\Phi_{rj}^{\sqrt{\Delta_{\Phi}^{\Psi}}}(y)\\
&=(S_{i}^{\Phi})^*\Delta_{1}\pi^{\varphi}(\langle x,y\rangle)S_{j}^{\Phi}\\
&=\varphi_{ij\Delta_{1}}(\langle x,y\rangle).
\end{align*}
Hence, we deduce
\begin{align*}
\langle[\Phi^{\sqrt{\Delta_{\Phi}^{\Psi}}}](x),[\Phi^{\sqrt{\Delta_{\Phi}^{\Psi}}}](x)\rangle=
[\varphi]_{\Delta_{1}}(\langle x,x\rangle)
=[\psi](\langle x,x\rangle)=\langle[\Psi](x),[\Psi]\rangle(x),
\end{align*}
and $[\Psi]\sim[\Phi^{\sqrt{\Delta_{\Psi}^{\Phi}}}]$. Assume that there exists another positive operator $T\oplus S\in(\pi^{\Phi}(\mathcal{M})')$ such that $[\Psi]\sim[\Phi^{\sqrt{T\oplus S}}]$. Then
$[\Phi^{\sqrt{\Delta_{\Psi}^{\Phi}}}]\sim[\Phi^{\sqrt{T\oplus S}}]$ and $[\varphi_{\Delta_{1}}]=[\varphi_{T}]$. Since the representation $\pi_{\Phi}$ is nondegenerate we deduce $\Delta_{1}=T$ and $\Delta_{\Psi}^{\Phi}=T\oplus S$.
\end{proof}
The operator $\Delta_{\Phi}^{\Psi}\in(\pi^{\Phi}(\mathcal{M})')$ is called the {\it Radon--Nikodym derivative } of $[\Psi]$ with respect to $[\Phi]$. We notice that if $\Delta_{\Psi}^{\Phi}=\Delta_{1}\oplus\Delta_{2}$ is the Radon--Nikodym derivative of $[\Psi]$ with respect to $[\Phi]$, then $\Delta_{1}$ is the Radon--Nikodym derivative of $[\psi]$ with respect to $[\varphi]$.

For $[\Phi]\in\mathcal{CP}_{n}(\mathcal{M},\mathscr{L}(\mathcal{H},\mathcal{K}))$ let
$$
\widetilde{[\Phi]}:=\{[\Psi]\in\mathcal{CP}_{n}(\mathcal{M},\mathscr{L}(\mathcal{H},\mathcal{K})):\,[\Psi]\sim[\Phi]\}
$$
Take $[\Psi],[\Phi]\in\mathcal{CP}_{n}(\mathcal{M},\mathscr{L}(\mathcal{H},\mathcal{K}))$. We use the notation $\widetilde{[\Psi]}\leq\widetilde{[\Phi]}$ when
$[\Psi]\preceq[\Phi]$. For $[\Phi]\in\mathcal{CP}_{n}(\mathcal{M},\mathscr{L}(\mathcal{H},\mathcal{K}))$ consider the set
$$
[0,\widehat{[\Phi]}]:=\{\widetilde{[\Psi]}:\,[\Psi]\in\mathcal{CP}_{n}(\mathcal{M},\mathscr{L}(\mathcal{H},\mathcal{K})):\,\widetilde{[\Psi]}\leq\widetilde{[\Phi]}\}
$$
and
$$
[0,I]_{\Phi}:=\{T\oplus N\in\pi^{\Phi}(\mathcal{M})':\,0\leq T\oplus N\leq I \}.
$$

By Theorem~\ref{RN} we deduce the existence of a linear map $\Upsilon:[0,\widehat{[\Phi]}]\to[0,I]_{\Phi}$ such that
\begin{gather}\label{m-1}
[0,\widetilde{[\Phi]}]\in\widetilde{[\Psi]}\mapsto\Upsilon(\widetilde{[\Psi]})=\Delta_{\Phi}([\Psi])\in[0,I]_{\Phi}.
\end{gather}

\begin{thm}\label{Iso}
Let $[\Phi]\in\mathcal{CP}_{n}(\mathcal{M},\mathscr{L}(\mathcal{H},\mathcal{K}))$. Then the map $\Upsilon:[0,\widetilde{[\Phi]}]\to[0,I]_{\Phi}$
is order-preserving isomorphism.
\end{thm}
\begin{proof}
By Theorem~\ref{RN}, the map $\Upsilon:[0,\widetilde{[\Phi]}]\to[0,I]_{\Phi}$
is well defined. Let $[\Psi_{1}],[\Psi_{2}]\in\mathcal{C}_{n}(\mathcal{M},\mathscr{L}(\mathcal{H},\mathcal{K}))$ and
$[\Psi_{1}]\preceq[\Psi], [\Psi_{2}]\preceq[\Psi]$, $\Delta_{\Phi}(\Psi_{1})=\Delta_{\Phi}(\Psi_{2})$. Then, by (\cite[Theorem~4.2.6]{J-0},
$[\Psi_{1}]\sim[\Psi_{2}]$. Hence $\widetilde{[\Psi_{1}]}=\widetilde{[\Psi_{2}]}$ and we prove that $\Upsilon$ is an injective linear map.
On the other hand, take an arbitrary operator $T\oplus N\in\pi^{\Phi}(\mathcal{M})'$ such that $0\leq T\oplus N\leq I$. Utilizing Lemma~\ref{com}
we have $[\Phi_{T\oplus N}]\preceq[\Phi]$ and $\widetilde{[\Phi_{T\oplus N}]}\leq\widetilde{[\Phi]}$. Hence, the map $\Upsilon$ is surjective. If
$[\Psi_{1}]\preceq[\Psi_{2}]\preceq[\Phi]$, then $[\psi_{1}]\preceq[\psi_{2}]\preceq[\varphi]$ and again by using \cite[Theorem~4.2.6]{J-0} we deduce that
$\Delta_{1\Phi}(\Psi_{1})\leq\Delta_{1\Phi}(\Psi_{2})$. Since the representation $\pi^{\Phi}(\mathcal{M})$ is nondegenerate and by employing Lemma~\ref{commut} we get $\Delta_{\Phi}(\Psi_{1})\leq\Delta_{\Phi}(\Psi_{2})$. If
$$
0\leq T_{1}\oplus N_{1}\leq T_{2}\oplus N_{2};\, T_{1}\oplus N_{1}, T_{2}\oplus N_{2}\in\pi^{\Phi}(\mathcal{M})',
$$
then $0\leq T_{1}\leq T_{2}\leq I$, $T_{1}, T_{2}\in \pi^{\varphi}(\mathscr{A})'$. Hence $[\varphi_{T_{1}}]\leq[\varphi_{T_{2}}]$ and therefore
$[\Phi_{T_{1}\oplus N_{1}}]\preceq[\Phi_{T_{2}\oplus N_{2}}]$.
\end{proof}
\begin{rem}
In \cite{J-01}, Joi\c{t}a proved a Radon--Nikodym type theorem for completely positive $n\times n$ matrices of maps from locally $C^{\star}$-algebra $\mathscr{A}$ to
the $C^{\star}$-algebra $\mathscr{L}(\mathcal{H})$ of all bounded linear bounded operators on a Hilbert space $\mathcal{H}$. If we regard a completely positive $n\times n$ matrix as a completely positive map $[\varphi]:M_{n}(\mathscr{A})\to M_{n}(\mathscr{L}(\mathcal{H}))$, then Theorem\ref{RN} can be considered as a Radon--Nikodym theorem for a module $[\varphi]$-map in the sense of recent works of $\varphi$-maps (see for instance \cite{As,B,J-2,SSum}).
\end{rem}

\textbf{Acknowledgement.} M. S. Moslehian was supported by a grant from Ferdowsi University of Mashhad (No. 2/43524). M. Pliev was financially supported by the Ministry of Education and Science of the Russian Federation (the Agreement number 02.A03.21.0008).


\medskip

\begin{thebibliography}{99}

\bibitem{ACM} \textsc{M. Amyari, M. Chakoshi and M. S. Moslehian}, \textit{Quasi-representations of Finsler modules over $C^*$-algebras}, J. Operator Theory \textbf{70} (2013), no. 1, 181--190.

\bibitem{As}
\textsc{M. D. Asadi}, \textit{Stinspring{'}s theorem for Hilbert $C^{\star}$-modules},  J. Operator Theory \textbf{62} (2009), no. 2, 235--238.

\bibitem{B}
\textsc{R. Bhat, G. Ramesh and K. Sumesh} \textit{Stinespring{'}s theorem for maps on Hilbert $C^*$-modules}, J. Operator Theory \textbf{68} (2012), no. 1, 173--178.

\bibitem{HH}
\textsc{J. Heo, J. P. Hong and U. C. Ji.} \textit{On KSGNS representation on Krein $C^*$-modules }, J. Math. Phys. \textbf{51} (2010), no. 5, 053504, 13p.

\bibitem{JJM}
\textsc{M. S. Moslehian, M. Joi\c{t}a and U. C. Ji}, \textit{KSGNS type construction for $\alpha$-completely positive maps on Krein $C^*$-modules}, Complex Anal. Oper. Theory \textbf{10} (2016), no. 3, 617--638.

\bibitem{F}
\textsc{M. Fragoulopoulou}, \textit{Topological algebras with involution}, Elsevier, 2005.

\bibitem{In}
\textsc{A. Inoue}, \textit{Locally $C^{*}$-algebras}, Mem. Faculty. Sci. Kyushu. Univ. Ser.A. \textbf{25} (1971), 197--235.

\bibitem{J-00}
\textsc{M. Joi\c{t}a}, \textit{Hilbert modules over locally $C^{*}$-algebras}, University of Bucharest Press, 2006.

\bibitem{J-0}
\textsc{M. Joi\c{t}a}, \textit{Completely positive linear on pro-$C^{*}$-algebras}, University of Bucharest Press, 2008.

\bibitem{J-01}
\textsc{M. Joi\c{t}a} \textit{A Radon--Nikodym theorem for completely multi-positive linear maps on pro-$C^{\star}$-algebras and its applications}, Publ. Math. Debrecen \textbf{72} (2008), no. 1-2, 55--67.

\bibitem{J-1}
\textsc{M. Joi\c{t}a} \textit{Covariant version of the Stinespring type theorem for Hilbert $C^*$-modules}, Cent. Eur. J. Math. \textbf{9} (2011), no. 4, 803--813.

\bibitem{J-2}
\textsc{M. Joi\c{t}a} \textit{Comparision of completely positive maps on Hilbert $C^*$-modules}, J. Math. Anal. Appl. \textbf{393} (2012), 644--650.

\bibitem{MP}
\textsc{I. N. Maliev and M. A. Pliev} \textit{A Stinespring-type representation for operators in Hilbert modules over local $C\sp *$-algebras} (Russian), Izv. Vyssh. Uchebn. Zaved. Mat. \textbf{2012} , no. 12, 51--58; translation in Russian Math. (Iz. VUZ) \textbf{56} (2012), no. 12, 43--49.

\bibitem{PT}
\textsc{M. A. Pliev and I. D. Tsopanov} \textit{On a representation of Stinespring type for $n$-tuples of completely positive maps in Hilbert $C^*$-modules} (Russian), Izv. Vyssh. Uchebn. Zaved. Mat. \textbf{2014}, no. 11, 42--49.

\bibitem{Ph}
\textsc{N. C. Phillips}, \textit{Inverse limits of $C^*$-algebras}, J. of Operator Theory, \textbf{19} (1988), 159--195.

\bibitem{S}
\textsc{M. Skeide}, \textit{A factorization theorem for $\varphi$-maps}, J. Operator Theory \textbf{68} (2012), no. 2, 543--547.

\bibitem{SSum}
\textsc{M. Skeide and K. Sumesh}, \textit{CP-H-Extendable maps between Hilbert Modules and CPH-semigroups}, J. Math. Anal. Appl. \textbf{414} (2014),  886--913.


\bibitem{St}
\textsc{W. F. Stinespring} \textit{Positive functions on $C^*$-algebras}, Proc. Amer. Math. Soc. \textbf{6} (1955), 211--216.

\bibitem{SUE}
\textsc{C. Y. Suen}, \textit{An $n\times n$ matrix of linear maps of a $C\sp *$-algebra}, Proc. Amer. Math. Soc. \textbf{112} (1991), no. 3, 709--712.



\end{thebibliography}
\end{document}